\documentclass{amsart}

\newtheorem{theorem}{Theorem}[section]
\newtheorem{definition}[theorem]{Definition}
\newtheorem{lemma}[theorem]{Lemma}
\newtheorem{conjecture}[theorem]{Conjecture}
\newtheorem{corollary}[theorem]{Corollary}

\usepackage{enumerate}

\title{Standard Examples as Subposets of Posets}

\author{Csaba Bir\'o}
\address{Department of Mathematics, University of Louisville, Louisville KY
40292, USA} \email{csaba.biro@louisville.edu}

\author{Peter Hamburger}
\address{Department of Mathematics, Western Kentucky
University, Bowling Green, KY 42101, USA}
\email{peter.hamburger@wku.edu}

\author{Attila P\'or}
\address{Department of Mathematics, Western Kentucky
University, Bowling Green, KY 42101, USA} \email{attila.por@wku.edu}

\begin{document}

\begin{abstract}
We prove that a poset with no induced subposet $S_k$ (for fixed
$k\geq 3$) must have dimension that is sublinear in terms of the
number of elements.
\end{abstract}

\maketitle

\section{Introduction}

Dushnik and Miller (1941) first studied order dimension,~
\cite{Dus-Mil-41}. Let $P$ be a poset. A set of its linear extensions
$\{L_1,\ldots,L_d\}$ forms a \emph{realizer}, if $L_1\cap\cdots\cap
L_d=P$. The minimum cardinality of a realizer is called the
\emph{dimension} of the poset $P$. This concept is also sometimes
called the \emph{order dimension} or the \emph{Dushnik-Miller
dimension of the partial order}.

The \emph{standard example} $S_n$ on $2n$ elements is the poset
formed by considering the $1$-element subsets, and the
$(n-1)$-element subsets of a set of $n$ elements, ordered by
inclusion. It is known that the dimension of $S_n$ is half of the
number of elements of $S_n$, that is, $\dim(S_n)=n$. It is also
known that there are posets of arbitrarily large dimension without
an $S_3$ subposet. For further study of the dimension of posets we
refer the readers to  the monograph \cite{Tro-CAPOS}, and to the
survey article \cite{Gra-Gro-Lov-HOC-Tro}.

We use the notation $|P|$ for the number of elements (of the ground
set) of the poset $P$, and we use the standard notation $[n]$ for
the set $\{1,2,\ldots,n\}$. The \emph{upset} (sometimes called
\emph{ideal}) of the element $x$ is the set $U(x)=\{y\in P: y>x\}$.
Similarly, the \emph{downset} of the element $x$ is $D(x)=\{y\in P:
y<x\}$. The symbol $x\|y$ is used  to denote that $x$ is
incomparable with $y$ in the poset.

Let $P$ and $Q$ be two posets. $Q$ is a \emph{subposet} of $P$, if
there is an induced copy of $Q$ in $P$. More precisely, $Q$ is a
subposet of $P$ if there is a bijection $f$ from a subset of the
ground set of $P$ to the ground set of $Q$ with the property that
$f(x)<f(y)$ if and only if $x<y$. If $P$ does not contain $Q$ as a
subposet, $P$ is called \emph{$Q$-free}. Note that some authors use
the term ``subposet'' in the ``non-induced'' sense. We conform to
the more standard usage, and we say ``an extension of $Q$ is a
subposet of $P$'' if we ever need the non-induced meaning.

Hiraguchi \cite{Hir-55} proved that for any poset with $|P|\geq 4$,
$\dim(P)\leq |P|/2$. Bogart and Trotter \cite{Bog-Tro-73} showed
that if $|P|\geq 8$, then the only extremal examples of Hiraguchi's
theorem are the standard examples $S_n$ with $\dim(S_n)=|S_n|/2$. A
natural question arises: if the dimension is just a bit less than
the
maximum, can we expect largely the same structure as in the extremal case?

\medskip

Bir\'o, F\"uredi, and Jahanbekam \cite{Bir-Fur-Jah-13} conjectured
the following.
\begin{conjecture}\label{conj:stability}
For every $t<1$, but sufficiently close to $1$ there is a $c>0$ and
$N$ positive integer, so that if $|P|\geq 2n\geq N$, and
$\dim(P)\geq tn$, then $P$ contains $S_{\lfloor cn\rfloor}$.
\end{conjecture}

Conjecture~\ref{conj:stability} was mentioned at the ``Problems in
Combinatorics and Posets'' workshop \cite{workshop} in 2012 and it
is listed as Problem \#1 of the workshop. Even the existence
of an $S_3$ was unclear.

\medskip

In this paper we show that a poset with no induced subposet $S_k$
(for fixed $k\geq 3$) must have dimension that is sublinear in terms
of the number of elements.

\medskip

A few more definitions are needed. The \emph{height} of a poset is
the size of a maximum chain. Height two posets are called bipartite.
A \emph{bipartition} of a bipartite poset $P$ is a partition $(A,B)$
of its ground set together with a fixed linear order on the parts,
such that $x<y$ in $P$ implies $x\in A$ and $y\in B$. Note that the
bipartition includes a linear ordering on both $A$ and $B$,
respectively.  They are not denoted to avoid clutter, but they are
important in the discussion.

A \emph{critical pair} $(x,y)$ is an ordered pair of elements of $P$
with the properties $x\|y$, $D(x)\subseteq D(y)$, and $U(y)\subseteq
U(x)$. A linear extension $L$ \emph{reverses} the critical pair
$(x,y)$, if $y<x$ in $L$. A set of linear extensions $\mathcal{L}$
\emph{reverses} $(x,y)$ if there is a linear extension in
$\mathcal{L}$ that reverses $(x,y)$. It is known that a set of
linear extensions is a realizer if and only if it reverses every
critical pair of $P$. This makes the investigation simpler in
bipartite posets, where all incomparable minimum--maximum pairs are
critical, and they are the only (interesting) critical pairs; the
remaining critical pairs have a very special structure and can be
usually handled in a simple way.

Let $P$ be a poset. A poset $Q$ is the \emph{dual} of $P$, if $Q$ is
defined on the same ground set, but every order is reversed, that
is, $x<y$ in $P$ iff $x>y$ in $Q$. Obviously, $\dim(P)=\dim(Q)$.

\section{Preliminaries}

\begin{definition}
Let $k\geq 3$ integer. Let ${\mathcal{F}}_k$ denote the set of
finite $S_k$-free posets.
\begin{align*}
D(n,k)&=\max\{\dim (P): P\in\mathcal{F}_k,|P|=n\}\\
\Delta(n,k)&=\max\{\dim (P): P\in\mathcal{F}_k,\text{ $P$ is
bipartite},|P|=n\}
\end{align*}
\end{definition}

The main goal of the paper is to prove the following theorem.
\begin{theorem}\label{thm:main}
\[
D(n,k)=o(n).
\]
\end{theorem}

The proof is in Section~\ref{sect:kimble}.\\

Let $\mathcal{H}$ be a hypergraph. A \emph{coloring} of
$\mathcal{H}$ is an assignment of positive integers (colors) to its
vertices (with no restrictions). A $\ell$-coloring is a coloring
with $\ell$ colors. A subset $H$ of $V(\mathcal{H})$ is
\emph{monochromatic}, if every edge $E\subseteq H$ receives the same
color in the coloring.

$K_n^k$ is used to denote the complete $k$-uniform hypergraph on $n$
vertices. We need the following version of Ramsey's Theorem.
\begin{theorem}\label{thm:ramsey}\cite{Wes-ITGT}
For all $k,q,\ell$ positive integers there exists an $N$ such that if $n\geq
N$, then every $\ell$-coloring of $K_n^k$ contains a monochromatic
set of size $q$.
\end{theorem}

The (hypergraph) Ramsey number $R(k,q,\ell)$ is the least $N$ in
Theorem~\ref{thm:ramsey}. For a fixed $k$, the function $R(k,q,k)$
is a function of one variable; the following function may be
regarded as its inverse.
\begin{definition}\label{crom}
Let $k\geq 3$ fixed. Let $r(n)$ denote the minimum size of a largest
monochromatic subset in a $k$-coloring of $K_n^k$.
\end{definition}
We use a simple corollary of Theorem~\ref{thm:ramsey}, that is, we
use that $r(n)\to\infty$ as $n\to\infty$.

Throughout this paper a positive integer $k\geq 3$ is fixed. The
purpose of this is that most of the time we assume that our posets
are $S_k$-free.

Let $P$ be an $S_k$-free bipartite poset with bipartition $(A,B)$.
We fix an ordering of the vertices of $A$. Let
$S=\{a_1,\ldots,a_k\}\subseteq A$, and assume that the indexing
preserves the ordering on $A$. Call an element $b\in B$ a
\emph{mate} of $a_i$, if $a_i\| b$, but $a_j<b$ for all $j\neq i$.
Clearly any $b\in B$ can not be a mate of more than one $a_i$, so
the set of mates of $a_1,\ldots,a_k$ form disjoint subsets of $B$.
The condition that $P$ is $S_k$-free means that there exits
$a_{i_0}$ that has no mate. In this case, we say that $i_0$ is a
valid color for $S$.

An \emph{upset based coloring} (or \emph{UB-coloring} for short) of
the $k$-element subsets of $A$ is such that assigns a valid color to
each subset. Note that UB-coloring is only defined in the context of
$S_k$-free bipartite posets with a fixed ordering on the minimal
elements.

\section{Bipartite posets}

We begin with a simple technical statement of probability that will
contain the key computation.
\begin{lemma}\label{lemma:prob}
Let $t,q$ be positive integers, $2\leq q$, $t\leq q$, and let $r\geq
t2^t\ln q$. Let $X=[x_{i,j}]$ be a random $r\times q$ binary matrix,
in which each entry is $1$ independently with probability $1/2$. Let
$E$ be the event that for all sequences $1\leq
j_1<j_2<\cdots<j_t\leq q$ and all integers $1\leq \ell\leq t$ there
is a row $i$ in $X$ with the property
$x_{i,j_1}=\cdots=x_{i,j_{\ell-1}}=x_{i,j_{\ell+1}}=\cdots=x_{i,j_t}=0$
and $x_{i,j_\ell}=1$.  Then $\Pr(E)>0$.
\end{lemma}
\begin{proof}
For a fixed sequence $s=(j_1,\ldots,j_t)$ and integer $\ell$, let
$E_{s,\ell}$ be the event that at least one row has the property.
Any given row has the property with probability $2^{-t}$, so
$\Pr(\overline{E_{s,\ell}})=(1-2^{-t})^r$. Hence
\begin{multline*}
\Pr(E) =\Pr(\cap E_{s,\ell}) =1-\Pr(\cup\overline{E_{s,\ell}}) \geq
1-\sum\Pr(\overline{E_{s,\ell}})
=1-t\binom{q}{t}(1-2^{-t})^r\\
>1-q^t e^{-r2^{-t}}
=1-e^{t\ln q-r2^{-t}}\geq 0,
\end{multline*}
where the last inequality follows from the condition on $r$.
\end{proof}

Recall that $k\geq 3$ is a fixed integer.
\begin{lemma}\label{lemma:linext}
Let $P$ be a bipartite poset with bipartition $(A,B)$. Let $q\geq
2$, and $Q=\{a_1,\ldots,a_q\}\subseteq A$ be a monochromatic set in
a UB-coloring, and assume that the indexing preserves the ordering
on $A$. Then there exists a set of linear extensions
$\mathcal{L}=\{L_1,\ldots,L_{q'}\}$ with $q'=2\lceil k2^k\ln
q\rceil$ that reverses every critical pair $(a_i,b)$ with $b\in B$.
\end{lemma}
\begin{proof}

$Q$ is a monochromatic set, denote its color by $\ell$. Let
$t=\max\{\ell-1,k-\ell\}$. Since $t<k$, we have that $q'/2>t2^t\ln
q$, so we can apply Lemma~\ref{lemma:prob} with $t$, $q$, and
$r=q'/2$. We conclude that there exists a matrix $X$ with the
property described in the lemma.

We construct the set $\mathcal{L}$ using $X$. For each row of $X$ we
construct two linear extensions. For a given binary row
$\underline{x}=(x_1,\ldots,x_q)$, first construct two permutations
$\sigma_1$ and $\sigma_2$ of $[q]$ as follows. For $\sigma_1$, first
list elements of $[q]$ for which the corresponding bit of
$\underline{x}$ is $1$ in order from left to right, then list the
elements corresponding to $0$ bits in order from left to right. For
$\sigma_2$, do the same, except list the elements of $[q]$ from
right to left.

In the next step, we construct a linear extension for each
permutation. We refer to the permutation as $\sigma$, which will be
first $\sigma_1$, then $\sigma_2$, thereby resulting two linear
extensions.  Let $U_i=\{b\in U(a_{\sigma(i)}):b\not\in
U(a_{\sigma(j)})\text{ for any }j<i\}$ for every $i\in [q]$. In
other words, $\{U_i\}$ forms a partition of the union of the upsets
of the elements of $Q$ such that any element that belongs to
multiple upsets will be placed into the first one, where the order
is determined by $\sigma$.  Let $R=P-Q-\bigcup_{i=1}^q U_i$. Then
define the linear extension with
\[
U_1>a_{\sigma(1)}>U_2>a_{\sigma(2)}>\cdots>U_q>a_{\sigma(q)}>R
\]
where the order within each $U_i$ and $R$ are arbitrary (e.g.\  we
can respect some predetermined order of the elements).

Repeating the process for every row we construct  the set
$\mathcal{L}$.

It remains to be shown that $\mathcal{L}$ reverses every critical
pair of the form $(a_i,b)$ with $b\in B$. Consider such a pair
$(a_i,b)$. Let
\[
M_1=\{a_m\in D(b): m<i\}\text{ and } M_2=\{a_m\in D(b): m>i\}.
\]
Since any $k$-subset $K=\{a_{\ell_1},\ldots,a_{\ell_k}\}$ with
$(\ell_1<\cdots<\ell_k)$ in which $a_{\ell_t}=a_i$ is colored
$\ell$, we know that $a_i$ cannot have a mate with respect to $K$.
Since $a_i\|b$, we have that either
\begin{enumerate}[i)]
\item\label{item:leftfew} $|M_1|\leq \ell-2$, or
\item\label{item:rightfew} $|M_2|\leq k-\ell-1$.
\end{enumerate}
In case \ref{item:leftfew}), find a row $\underline{x}$ of $X$ such
that $x_j=0$ for all $j$ for which $a_j\in M_1$, and $x_i=1$. Such a
row exists, because $|M_1|+1\leq t$ (we need to force a few more
zeros in the row to directly use the lemma, if $|M_1|<t-2$). The
first linear extension corresponding to this row places $a_i$ over
$b$. In case \ref{item:rightfew}), find a row $\underline{x}$ of $X$
such that $x_j=0$ for all $j$ for which $a_j\in M_2$, and $x_i=1$.
The second linear extension corresponding to this row places $a_i$
over $b$.
\end{proof}

\begin{lemma}\label{lemma:removable}
Let $P$ be a bipartite $S_k$-free poset with a bipartition $(A,B)$,
and $|A|\geq n$, where $n$ is such that $r(n)\geq 2$, (see $r(n)$ in
Definition~\ref{crom}). Then there exist $c=c(k)$ such that for all
$q=2,3,\ldots,r(n)$ there exists $Q\subseteq A$ with $|Q|=q$ such
that
\[
\dim(P)\leq\dim(P-Q)+c\ln q.
\]
\end{lemma}
\begin{proof} We show that $c=3k2^k$ works. Let $2\leq q\leq r(n)$ be an
arbitrary integer. Consider a UB-coloring of the $k$-subsets of $A$.
By Ramsey's Theorem there is a monochromatic subset $Q$ with
$|Q|=q$. Let $P'=P-Q$, and let $\mathcal{L'}$ be a realizer of $P'$.
By Lemma~\ref{lemma:linext} there exists a set of linear extensions
$\mathcal{L''}$ of size $2\lceil k2^k\ln q\rceil$ that reverses
every critical pair in $P$ of the form $(a,b)$, where $a\in Q$, and
$b\in B$. Now consider the set $M$, the set of all minimal elements
of $P$. Add a linear extension $L$ in which we reverse the order of
elements of $M$ (compared to a fixed element of, say,
$\mathcal{L''}$) but the rest of the elements are inserted
arbitrarily.

Now any critical pair $(x,y)$ is reversed in $\mathcal{L'}$ or
$\mathcal{L''}$, unless $y\in Q$, but the only way that can happen
if both $x,y$ are minimal, so they are reversed in $L$, if they
haven't been before.

We have shown that $\dim(P)\leq\dim(P-Q)+2\lceil k2^k\ln q\rceil+1
\leq\dim(P-Q)+3k2^k\ln q$.
\end{proof}

\begin{theorem}\label{thm:bipartite} Let $k\geq 3$ integer.
\[
\Delta(n,k)=o(n)
\]
\end{theorem}
\begin{proof}
Fix $\epsilon>0$. Let $c$ be as in Lemma~\ref{lemma:removable}, and
let $q\geq 2$ be such that $c\ln q/q\leq \epsilon/2$. Furthermore,
let $N$ be such that $r(\lceil N/2\rceil)\geq q$. We show that for
all $\ell\geq 0$ integer,
\[
\Delta(N+\ell q,k)\leq\Delta(N,k)+\ell\frac{\epsilon}{2}q.
\]

To show this we use induction on $\ell$. It is obvious for the case
$\ell=0$. Assume $\ell\geq 1$. Consider an $S_k$-free bipartite
poset $P$ on $N+\ell q$ elements with $\dim(P)=\Delta(N+\ell q,k)$.
Without loss of generality we may assume that $P$ has a bipartition
$(A,B)$ with $|A|\geq \lceil N/2\rceil$, for otherwise we may
consider the dual of $P$ instead.

By Lemma~\ref{lemma:removable}, $P$ has a subset $Q$ of size $q$,
such that $\dim(P)\leq\dim(P-Q)+c\ln q$. Hence
\begin{multline*}
\Delta(N+\ell q,k) \leq\dim(P-Q)+c\ln q
\leq \Delta(N+(\ell-1)q,k)+c\ln q\\
\leq \Delta(N,k)+(\ell-1)\frac{\epsilon}{2}q+c\ln q \leq
\Delta(N,k)+(\ell-1)\frac{\epsilon}{2}q+\frac{\epsilon}{2}q
=\Delta(N,k)+\ell\frac{\epsilon}{2}q.
\end{multline*}
This, with the fact that $\Delta(n,k)$ is a monotone increasing
sequence of $n$, finishes the proof.
\end{proof}

\section{Kimble splits and general posets}\label{sect:kimble}

We need the notion of a ``split'' of a poset. Kimble \cite{Kim-73}
introduced this notion. We only need a special version of his
definition, and a special case of his theorem; we only mention
those.

\begin{definition}
Let $P$ be a poset with ground set $\{x_1,\ldots,x_n\}$. The
\emph{Kimble split} of $P$ is the poset on the ground set
$\{x_1',x_1'',\ldots,x_n',x_n''\}$ with $x_i'\leq x_j''$ if and only
if $x_i\leq x_j$ in $P$.
\end{definition}

\begin{theorem}\cite{Kim-73}\label{thm:splitdim}
Let $P$ be a poset and $Q$ be its Kimble split. Then
\[
\dim(P)\leq\dim(Q)\leq\dim(P)+1
\]
\end{theorem}

We also need the following simple lemma.

\begin{lemma}\label{lemma:splitnoSk}
Let $P$ be an $S_k$-free poset and $Q$ be its Kimble split. Then $Q$
is $S_k$-free.
\end{lemma}
\begin{proof}
Suppose $Q$ has an $S_k$ subposet, call the set of its vertices $S$.
If all $2k$ elements of $S$ come from distinct elements of $P$, then
they formed an $S_k$ in $P$. Otherwise there is a pair $a',a''\in S$
such that $a'<a''$ in $Q$, and they are the split versions of the
original vertex $a$ of $P$. Since $k\geq 3$, there is $b',c''\in S$
with $b'>a'$, and $c''<a''$, and $b'\| c''$ in $Q$. Clearly $b'$ and
$c''$ came from distinct elements of $P$, because they are
incomparable, and these elements are also distinct from $a$; call
them $b$ and $c$ respectively. Due to the definition of the Kimble
split, $b<a<c$ in $P$, but then $b'<c''$ in $Q$, a contradiction.
\end{proof}

\subsection{Proof of Theorem~\ref{thm:main}}

Lemma~\ref{lemma:splitnoSk} and the first inequality of
Theorem~\ref{thm:splitdim} imply that
\[
D(n,k)\leq\Delta(2n,k),
\]
so the theorem is a consequence of Theorem~\ref{thm:bipartite}.

We end the discussion with a corollary in the style of (the still
open) Conjecture~\ref{conj:stability}.

\begin{corollary}
For all $k\geq 3$ integer, and $t<1$ there is an $N$ integer, so
that if $|P|\geq 2n\geq N$, and $\dim(P)\geq tn$, then $P$ contains
$S_k$.
\end{corollary}


\providecommand{\bysame}{\leavevmode\hbox to3em{\hrulefill}\thinspace}
\providecommand{\MR}{\relax\ifhmode\unskip\space\fi MR }
\providecommand{\MRhref}[2]{%
  \href{http://www.ams.org/mathscinet-getitem?mr=#1}{#2}
}
\providecommand{\href}[2]{#2}

\end{document}